\documentclass[12, reqno] {amsart}
\usepackage{amsmath}
\usepackage{amscd,amsthm,amsfonts,amsopn,amssymb,verbatim}
\numberwithin{equation}{section}
\usepackage{epsfig}
\usepackage{graphicx,epsfig}
\usepackage{amsmath,amsthm,amssymb,graphicx}
\usepackage{color}
\usepackage{times}
\newtheorem{theorem}{Theorem} 

\newtheorem{proposition}[theorem]{Proposition}

\newtheorem{lemma}[theorem]{Lemma}

\theoremstyle{remark}

\def\be{\begin{equation}}
\def\ee{\end{equation}}
\def\vp{\varphi}
\def\ve{\varepsilon}

\begin{document}
\Large

\title{On uniformly bounded basis in spaces of holomorphic functions}
\author{Jean Bourgain} \address{(J. Bourgain) Institute for Advanced Study, Princeton, NJ 08540}
\email{bourgain@math.ias.edu}
\thanks{This work was partially supported by NSF grants DMS-1301619}

\begin{abstract}
The main result of the paper is the construction of explicit uniformly bounded basis in the spaces of complex
homogenous polynomials on the unit ball of $C^3$, extending an earlier result of the author in the $C^2$ case.
\end{abstract}
\maketitle

\section
{Introduction}

This Note originates from the recent paper \cite {S} that was kindly brought to the author's attention.
In the introductory part of \cite {S}, the following two problems are put forward.

\medskip
\noindent
{\bf Problem 1.}
{\sl Let $B_d$ be the closed unit ball in $\mathbb C^d$ and $S_d=\partial B_d=\{\zeta \in \mathbb C^d$;
$\Vert \zeta\Vert =\langle \zeta, \zeta\rangle^{\frac 12} =1\}$ the unit sphere.
Denote for $N=1, 2, 3\ldots$
$$
\mathcal P_N=\mathcal P_N^{(d)}= \text{ span}
\{\zeta^\alpha =\zeta_1^{\alpha_1}\cdots\zeta_d^{\alpha_d}; \alpha_i\geq 0 \text { and } |\alpha|=\alpha_1+\cdots+
\alpha_d=N\}\eqno{(1.1)}
$$
the space of degree $N$ homogenous polynomials.

\noindent
Do the spaces $\mathcal P_N$ have orthonormal basis that are uniformly bounded in $L^\infty (B_d)$?}
~{\hfill$\square$}
 
\medskip
\noindent
{\bf Problem 2.}
{\sl Does the Hilbert space of homomorphic polynomials on $S_d$ admit a uniformly bounded orthonormal basis?
Same question for smooth strictly pseudo-convex domain $\Omega \subset\mathbb C^d$.}
\hfill$\square$

\medskip

The first problem was solved affirmatively in \cite {B} if $d=2$, hence also answering Problem 2 for $d=2$.
Extending the approach from \cite {B} to $d>2$ turns out to be not straightforward.
In this paper we will give an construction for $d=3$ which potentially may be generalized to higher dimension, though this could require
additional work.
On the other hand, one can provide an affirmative solution to Problem 2, without going through Problem 1.
The core of the argument is a general result on orthonormal basis, proven in \cite{O-P} (and going back to a construction
due to A.~Olevskii, \cite {Ol}), which seems little known outside the experts' circle.

\section
{Uniformly bounded orthonormal basis}

We start with the following result (Theorem 2 in \cite {O-P}).

\begin{proposition}
Let $E$ be a separable linear subspace of a Hilbert space $L^2(\mu)$, $\mu$ a probability measure.
Then $E$ admits an orthonormal basis consisting of uniformly bounded functions, if and only if

\begin{itemize}
\item [(i)] $E\cap L^\infty(\mu)$ is dense in $E$ in the $L^2(\mu)$-norm
\item [(ii)] $E\cap \{f \in L^\infty (\mu):\Vert f\Vert_\infty \leq 1\}$ is not a totally bounded subset of $L^2(\mu)$.
\end{itemize}
\end{proposition}

Let $\Omega$ be a smooth strictly pseudo-convex domain and $L^2(\mu)=L^2(\partial\Omega, \sigma)$ with $\sigma$ the normalized surface
measure of $\partial\Omega$.
We take for $E$ the restriction to $\partial\Omega$ of the linear space of holomorphic polynomials.
Hence condition (i) is obviously satisfied.
For $\Omega= B_d$, results from \cite {R-W} and \cite {K} provide a sequence of elements $p_N\in\mathcal P_N^{(d)}$ such that $\Vert
p_N\Vert_2=1$ and $\Vert p_N\Vert _\infty =C_d$, taking care of condition (ii).

More generally, for $\Omega\subset\mathbb C^d$ smooth and strictly pseudo-convex, a result due to E.~Low \cite{Lo}
asserts in particular that if $\phi>0$ is a continuous function on $\partial\Omega$, then for all $\ve>0$, there exists $g\in A(\Omega)$
(the algebra of holomorphic functions on $\Omega$ that extend continuously to $\bar \Omega$) such that $|g|\leq\phi$ on $\partial\Omega$
and
$$
\sigma(\{\zeta\in\partial\Omega; |g|\not= \phi\})<\ve.\eqno{(2.1)}
$$
Hence $A(\Omega)\cap\{f\in L^\infty(\partial\Omega); \Vert f\Vert_\infty\leq 1\}$
is not totally bounded in $L^2(\partial\Omega)$.
Next, we are invoking a result of Henkin \cite{H}, Kerzman \cite {K} and Lieb \cite {Li} according to which elements of $A(\Omega)$ can
be approximated uniformly on $\bar\Omega$ by functions holomorphic on a neighborhood of $\bar\Omega$, hence by holomorphic polynomials.
Thus in conclusion, we again get condition (ii) satisfied.  We proved

\begin{proposition}
If $\Omega \subset \mathbb C^d$ is a smooth strictly pseudo-convex domain, then the holomorphic polynomials on $\partial\Omega$ admit a
uniformly bounded orthonormal basis.
\end{proposition}

\section
{Construction of uniformly bounded orthonormal basis in $\mathcal P_N^{(3)}$}

Answering a question of W.~Rudin, the author proved in \cite {B} that for $d=2$, the spaces $\mathcal P_N^{(2)}$ admit orthonormal basis
that are uniformly bounded in $L^\infty(B_2)$.
In this section, we revisit this construction, seeking for a higher dimensional extension and succeed in doing so for $d=3$.

We believe that (unlike \cite{B}) this approach may be generalizable and will indicate how.

Recall that for $d=2$, the basis are explicit and simple to describe.
More specifically, we introduce in \cite{B} polynomials $\big(\zeta =(z, w)\big)$
$$
\vp_k(\zeta)= (N+1)^{-1/2} \sum^N_{j=0} \sigma_j \,  e^{2\pi i \frac {jk}{N+1}} \ \frac {z^jw^{N-j}}{\Vert z^j w^{N-j}\Vert_2}\eqno{(3.1)}
$$
in $\mathcal P_N$, where $\{\sigma_j\}^N_{j=0}$ is a suitable unimodular sequence, which is taking to be the classical $\pm1$-valued
Rudin-Shapiro sequence
$$
\sigma_j =(-1)^{a_j} \ \text { with } \ a_j=\sum \ve_j\ve_{i+1}\eqno{(3.2)}
$$
and $\ve_i$ the digits in the binary expansion of $n$.

Certainly, there are other choices since the only relevant property of $\{\sigma_j\}$ is bound
$$
\max_\theta \Big|\sum_{j\in I}\sigma_j e(j\theta)\Big|\leq C|I|^{\frac 12}\eqno{(3.3)}
$$
where $I\subset\mathbb Z$ is an arbitrary interval (we use the notation $e(\theta)=e^{2\pi i\theta}$).

\begin{proposition}
The spaces $\mathcal P_N^{(3)} =\text {\rm span}[\zeta_1^{j_1} \zeta_2^{j_2}\zeta_3^{N-j_1-j_2}; 
j_1, j_2\geq 0,  j_1 +j_2\leq N]$
admit uniformly bounded basis.
\end{proposition}

We need some notation.
Let us assume $N$ odd and define
$$
\begin{aligned}
\Delta &= \{(j_1, j_2)\in\mathbb Z^2; j_1, j_2\geq 0, j_1+j_2\leq N\}\\
\Delta_0&= \{(0, j); 0\leq j\leq N\}\\
\Delta' &= \Big\{(j_1, j_2)\in \Delta; j_1\geq 1, 0\leq j_2\leq \frac {N-1}2\Big\}\\
\Delta''&= \Big \{(j_1, j_2)\in\Delta; j_1\geq 1, \frac {N+1}2 \leq j_2\leq N-1\Big\}
\end{aligned}
$$

\centerline{
\begin{minipage}{1.5 in}
\centerline
{\hbox{\includegraphics[width = 3.5in]{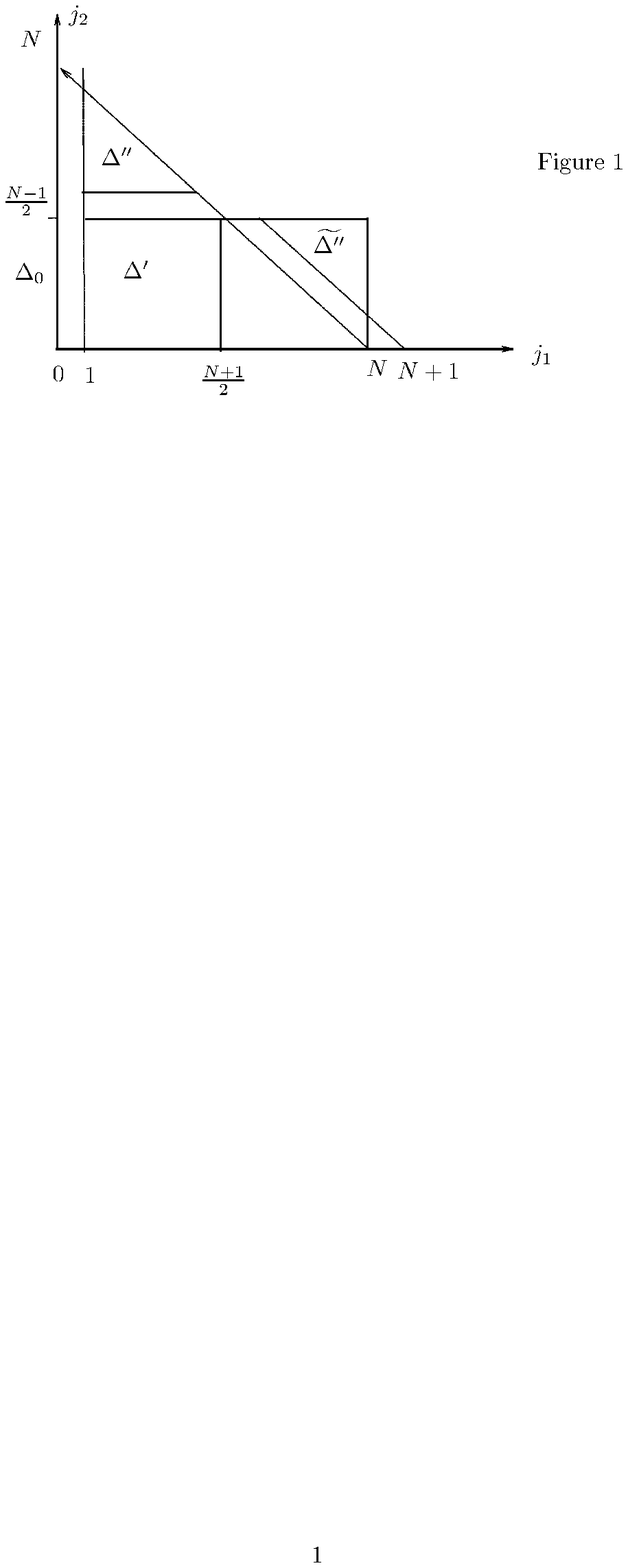}}}
\end{minipage}}

Hence $\Delta =\Delta_0 \cup\Delta'\cup\Delta''$,
$$
\begin{aligned}
&|\Delta|= \frac 12(N+1)(N+2)=D\\
&|\Delta'|+ |\Delta''| =\frac {N+1} 2\cdot N.
\end{aligned}
$$
Write the following orthogonal decomposition of $\mathcal P_N$
$$
\mathcal P_N=X\oplus Y
$$
with
$$
\begin{aligned}
X&=\text{\,\rm span\,}[\zeta^\alpha; (\alpha_1, \alpha_2)\in \Delta'\cup\Delta'']\\
Y&= \text {\, span\,}  [\zeta^\alpha, \alpha_1=0]
\end{aligned}
$$
Invoking Proposition 1.4.9 from \cite{R}, if $\alpha =(\alpha_1, \ldots, \alpha_d)$
$$
\int_{\partial B_d} |\zeta^\alpha|^2 d\sigma =\frac {(d-1)!\, \alpha_1!\cdots \alpha_d!}{(d-1+\alpha_1+\cdots+ \alpha_d)!}
$$
and for $d=3$
$$
\Vert\zeta_1^{j_1} \zeta_2^{j_2} \zeta_3^{N-j_1-j_2}\Vert^2_{L^2(\partial B_3)}=
\frac {2j_1!j_2!(N-j_1-j_2)!}{(N+2)!}.\eqno{(3.5)}
$$
Let us first consider the space
$$
Y=[\zeta_2^{j_2} \zeta_3^{N-j_2}; 0\leq j_2\leq N].
$$

Going back to (3.1), define for $k=0, \ldots, N$ the orthogonal system
$$
\psi_k (\zeta) =(N+1)^{-\frac 12} \sum^N_{j_2=0} \sigma_j \, e\Big(\frac {j_2k}{N+1}\Big) \ \frac {\zeta_2^{j_2} \zeta_3^{N-j_2}}
{\Vert \zeta_2^{j_2}\zeta_3^{N-j_2}\Vert_{L^2(\partial B_3)}}.
$$
Since by (3.4)
$$
\Vert\zeta_2^{j_2}\zeta_3^{N-j_2}\Vert_{L^2(\partial B_3)}=\Big[\frac {2 j_2!(N-j_2)!}{(N+2)!}\Big]^{\frac 12}
$$
and
$$
\Vert\zeta_2^{j_2} \zeta_3^{N-j_2}\Vert_{L^2(\partial B_2)}= \Big[\frac {j_2!(N-j_2)!}{(N+1)!}\Big]^{\frac 12}
$$
it follows from the $d=2$ construction that
$$
\psi_k=\Big(\frac N2+1\Big)^{\frac 12} \vp_k\eqno{(3.6)}
$$
with $\{\vp_k\}$ a uniformly orthonormal basis on $L^2(\partial B_2)$.

In particular, we have
$$
\Vert \psi_k\Vert_\infty \leq cN^{\frac 12}.\eqno{(3.7)}
$$
Assume we constructed a uniformly bounded orthonormal basis $f_1, \ldots, f_{N{\frac {N+1}2}}$ for the space $X$.
One can then apply Olevskii's absorption scheme \cite {Ol} to produce a uniformly bounded basis for $\mathcal P_N^{(3)}$.
We recall the construction.
Define for $k=0, \ldots, D-1$
$$
g_k =a_{k, 0}\, \psi_0+\cdots+ a_{k, N} \, \psi_N+ a_{k, N+1} \, f_1+\cdots+ a_{k, D-1} \, f_{\frac {N(N+1)}{2}}\eqno{(3.8)}
$$
where $A=(a_{k, \ell})_{0\leq k, \ell <D}\in O(D)$ will be specified next.

Let
$$
2^m\leq D< 2^{m+1}
$$
and
$$
\begin{cases}
a_{k, \ell}&= H^{(m)} (k, \ell)  \ \text { if } \ 0\leq k, \ell <2^m\\
a_{k, \ell}&= 0 \ \text { for } \ 0\leq k< 2^m, \ell\geq 2^m\\
a_{k, \ell}&= \delta_{k, \ell} \ \text { for } \ 2^m\leq k< D
\end{cases}
\eqno{(3.9)}
$$
where $H^{(m)}=H$ is the discrete Haar system on $\{0, 1, \ldots, 2^m -1\}$.
Thus
$$
\begin{cases}
H(e_0)&= \frac 1{\sqrt {2^m}} \ 1_{[0, 2^m[}\\
H(e_1)&= \frac 1{\sqrt {2^m}} (1_{[0, 2^{m-1}[}- 1_{[2^{m-1}, 2^m[})\\
H(e_2)&= \frac 1{\sqrt{2^{m-1}}} ( 1_{0, 2^{m-2}}[ - 1_{[2^{m-2}, 2^{m-1}[})\\
H(e_3)&=\frac 1{\sqrt{2^{m-1}}} (1_{[2^{m-1}, 2^{m-1} + 2^{m-2}[} - 1_{[2^{m-1}+ 2^{m-2}, 2^m[})\\
\text{etc. }&
\end{cases}
\eqno{(3.10)}
$$

Clearly $A\subset \mathcal O(D)$.
In view of (3.7)-(3.20), one easily verifies that
$$
\begin{aligned}
\Vert g_k\Vert_\infty &\leq C\Big(\frac 1{\sqrt{2^m}}+\frac {\sqrt 2}{\sqrt{2^m}}+\cdots+ \frac {(\sqrt 2)^{^2{\log N}}}
{\sqrt{2^m}} \Big) \sqrt N+C\\
& < C\frac N{\sqrt D} + C> C.
\end{aligned}
$$
Hence it remains to construct a uniformly basis for $X$.

Going back to Figure 1, let $\widetilde{\Delta''}$ be a triangle congruent to $\Delta''$ with vertices at $\big(\frac {N+3}2, \frac {N-1} 2\big)$,
$(N, 1), \big(N, \frac {N-1}2\big)$ and let $T:\Delta''\to \widetilde{\Delta''}$ be an affine map.
For $k=(k_1, k_2)\in \Delta''$, denote $\tilde k= Tk$.
Note that
$$
\mathbb Z^2 \cap (\Delta'\cup \widetilde{\Delta''}) = \{ 1, \ldots, N\}\times \Big\{0, \ldots, \frac {N-1}2\Big\}.\eqno{(3.11)}
$$
Define for $j=(j_1, j_2)\in \Delta'$
$$
\eta_j =\frac 1{\sqrt{\frac {N(N+1)}2}} \Big(\sum_{k\in\Delta'} u_k \, e\Big(\frac {j_1 k_1}N+\frac {j_2 k_2}{\frac {N+1}2}\Big) 
 e_k
+\sum_{k\in\Delta''} u_k \, e\Big(\frac {j_1 \tilde k_1}{N} +\frac {j_2 \tilde k_2}{\frac {N+1}2}\Big) e_k\Big)
\eqno{(3.12)}
$$
and for $j\in \Delta''$
$$
\eta_j= \frac 1{\sqrt{N\frac {N+1}2}} \Big(\sum_{k\in\Delta'} u_k \, e\Big(\frac {\tilde {j_1} k_2}{N}
+ \frac {\tilde j_2 k_2}{\frac {N+1}2}
\Big) \, e_k +\sum_{k\in\Delta''} u_k e\Big(\frac {\tilde j_1\tilde k_i}N+\frac {\tilde j_2\tilde k_2}{\frac {N+1} 2}\Big)\Big)\eqno{(3.13)}
$$
where
$$
e_j=\frac {\zeta_1^{j_1 } \zeta_2^{j_2} \zeta_3^{N-j_1-j_2}}{\Vert\zeta_1^{j_1}\zeta_2^{j_2} \zeta_3^{N-j_1-j_2} \Vert_{L^2(\partial B_3)}}
$$

\medskip

\noindent
and $(u_k)=(u_{k_1k_2})$ will be some unimodular sequence.

We first verify that $(\eta_j)_{j\in\Delta'\cup\Delta''}$ is an orthonormal system.

By orthogonality and (3.11), if $j, j'\in\Delta'$
$$
\begin{aligned}
\langle \eta_j, \eta_{j'}\rangle &=\frac 1{N\frac {N+1}2}\Big\{\sum_{k\in\Delta'} e\Big(\frac {j_1-j_1'}N k_1+\frac {j_2- j_2'}{\frac {N+1} 2}
k_2\Big)+ \sum_{k\in\widetilde{\Delta''}} e\Big(\frac {j_1-j_1'}N k_1
+\frac {j_2-j_2'}{\frac {N+1} 2} k_2\Big)\Big\}\\
&=\frac 1{N\frac {N+1}2} \ \sum_{\substack{1\leq k_1\leq N\\ 0\leq k_2\leq \frac {N-1}2}} \
e\Big(\frac {j_1-j_1'}N k_1 +\frac {j_2-j_2'}{\frac {N+1} 2} k_2 \Big) =\delta_{j, j'}
\end{aligned}
$$
and similarly for $j, j'\in \Delta''$.

For $j\in \Delta', j'\in\Delta''$, we obtain
$$
\frac 1{N \frac {N+1}2} \ \sum_{\substack{1\leq k_1<N\\ 0\leq k_2 \leq \frac {N-1}2}}
e\Big(\frac {j_1-(\widetilde {j'})_1}N k_1+ \frac {j_2-(\widetilde{j'})_2} {\frac {N+1}2} k_2\Big)
=\delta_{j, \widetilde{j'}} =0.
$$
Hence $(\eta_j)_{j\in\Delta'\cup\Delta''}$ is a basis for $X$.

Remains to introduce the sequence $u_k$.
This is the main novel input compared with \cite{B} (Rudin-Shapiro sequence based constructions do not seem to fit our
purpose).

Define
$$
u_{k_1, k_2} =e\big(\sqrt 2(k_1^2+k_2^2)\big). \eqno{(3.14)}
$$
The only role  of $\sqrt 2$ is its diophantine property
$$
\min_{x\in\mathbb Z, x\not=0} |x| \, \Vert x\sqrt 2\Vert>c>0\eqno{(3.15)}
$$
$(\Vert \ \Vert$ the distance to the nearest integer).
Since $\sqrt 2$ is a quadratic irrational, it has a periodic and hence bounded sequence of partial quotients, hence
(3.15).

We will rely on the following two estimates, which also explain the role of (3.15).

\begin{lemma}
Let $I_1, I_2$ be two arbitrary intervals of size $M_1, M_2$ and centers $c_1, c_2$.
Rather than summing over $I_1\times I_2$ we introduce a mollification, considering a smooth, symmetric, compactly supported
bump function $0\leq \rho\leq 1$ and a weight $\rho\big(\frac {k_1-c_1}{M_1}\big)\rho\big(\frac {k_2-c_2}{M_2}\big)$.
The following inequalities hold
$$
\max_{\psi_1, \psi_2\in\mathbb R} \Big|\sum_k \rho\Big(\frac {k_1-c_1}{M_1}\Big)
\rho\Big(\frac {k_2-c_2}{M_2}\Big) u_k\Big| \lesssim \sqrt{M_1M_2}
\leqno{(3.16)}
$$
and
$$
\max_{\psi_1, \psi_2\in\mathbb R}\Big|\sum_k\rho\Big(\frac {k_1-c_1}{M_1}\Big) \rho\Big(\frac {k_2-c_2}{M_2}\Big)
u_{k_1, N-k_1-k_2}\Big|\lesssim \sqrt{M_1M_2}.\leqno{(3.17)}
$$
\end{lemma}

\begin{proof}
\hfill

\noindent
\boxed{3.16}\hfill\break
\noindent
Denoting $S=\sum_k \rho\big(\frac {k_1-c_1}{M_1}\big) \rho\big(\frac {k_2-c_2}{M_2}\big)
e(k.\psi) u_{k}$, we obtain by squaring
$$
\begin{aligned}
|S|^2 &=\sum_{k, k'} \rho\Big(\frac {k_1-c_1}{M_1}\Big) \rho \Big( \frac {k_1'-c_1}{M_1}\Big)\rho
\Big(\frac {k_2-c_2}{M_2}\Big)\rho \Big(\frac{k_2'-c_2}{M_2}\Big)\\
&e\big((k-k').\psi\big) \, e\Big(\big(\sqrt 2\ ( k_1-k_1')(k_1+k_1')+(k_2-k_2')(k_2+k_2')\big)\Big)\\
&=\sum_{k, k'} \rho\Big(\frac {k_1}{M_1}\Big)\rho\Big(\frac {k_1'}{M_1}\Big)\rho\Big(\frac {k_2}{M_2}\Big)
\rho\Big(\frac {k_2'}{M_2}\Big) e\big((k-k')\psi'\big)\\
& e\big(\sqrt 2((k_1-k_1')(k_1+k_1')+ (k_2-k_2')(k_2+k_2'))\big)
\end{aligned}
$$
with $\psi'=\psi+2\sqrt 2 c$.
Making a change of variables $\ell =k -k'$, $\ell'= k+k'$, we obtain
$$
\sum_{\ell, \ell'} \rho\Big(\frac {\ell_1+\ell_1'}{2M_1}\Big) \rho\Big(\frac {\ell_1-\ell_1'}{2M_1}\Big)\rho
\Big( \frac {\ell_2-\ell_2'}{2M_2}\Big) \rho\Big(\frac {\ell_2-\ell_2'}{2M_2}\Big)
e(\ell.\psi') \, e\big(\sqrt 2 (\ell_1\ell_1'+\ell_2\ell_2')\big)
$$ 
which may be bounded by expressions of the form
$$
\Big[\sum_{\ell_1, \ell_1'}\rho_1 \Big(\frac {\ell_1}{L_1}\Big) \rho_2 \Big(\frac {\ell_1'}{L_1'}\Big)
e(\ell_1\psi_1') \, e(\sqrt 2\ell_1 \ell_1') \Big]
\ \Big[\sum_{\ell_2, \ell_2'} \rho_1 \Big(\frac {\ell_2}{L_2}\Big)\rho_2 \Big(\frac {\ell_2'}{L_2'}\Big)
e(\ell_2 \psi_2') e(\sqrt 2\ell_2\ell_2')\Big]
\eqno{(3.18)}
$$
where $L_1, L_1'\lesssim M_1, L_2, L_2' \lesssim M_2$.

We estimate each of the factors of (3.18).
We have
$$
\Big|\sum_{\ell, \ell'} \rho_1\Big(\frac \ell L\Big) \rho_2 \Big(\frac {\ell'}{L'}\Big)
e(\ell\psi') e(\sqrt 2 \ell_1 \ell_2')\Big| \leq \sum_\ell \rho_1
\Big(\frac \ell L\Big) \Big|\sum_{\ell'} \rho_2 \Big(\frac {\ell'}{L'}\Big)
 e(\sqrt 2\ell\ell')\Big|.
\eqno{(3.19)}
$$
By Poisson summation, the inner sum in (3.19) equals
$$
\begin{aligned}
&L'\sum_{k\in Z} \hat\rho_2\big(L'(k-\sqrt 2 \ell)\big)\leq \\
&L' \sum_{k\in\mathbb Z} \frac 1{(L')^2|k -\sqrt 2\ell|^2+1} \qquad \text { (since $\rho_2$ is smooth)} \\
& \leq \frac C{L'} +\frac 1{L' \Vert\sqrt 2 \ell\Vert^2+\frac 1{L'}}
\end{aligned}
$$
and summation over $\ell $ gives
$$
\frac C{L'} \sum_{\ell\lesssim L} \frac 1{\Vert\sqrt 2\ell\Vert^2 +\frac 1{(L')^2}}. \eqno{(3.20)}
$$

At this point, we use (3.15).
Clearly (3.15) implies
$$
|\{\ell \leq L; \Vert\sqrt 2\ell\Vert \sim 2^{-s} \}\Vert \lesssim \frac L {2^s} +1
$$
which permits to bound (3.20) by
$$
\frac C{L'} \sum_{s, 2^s\leq L'} \Big(\frac L{2^s}+1\Big)4^s\lesssim C(L+L').
$$
Hence (3.18) is bounded by $M_1.M_2$, proving (3.16).

\medskip

\noindent
\boxed{(3.17)}

Now
$$
S=\sum_k \rho\Big(\frac {k_1-c_1}{M_1}\Big) \rho\Big(\frac {k_2-c_2}{M_2}\Big) \, e(k.\psi)
\, e\big(\sqrt 2\big( (N-k_1-k_2)^2 +k^2_1\big)\big)
$$
hence
$$
|S| =\Big|\sum_k \rho\Big(\frac {k_1-c_1}{M_1}\Big)\rho\Big(\frac {k_2-c_2}{M_2}\Big)
\,  e (k.\psi') \, e\big(\sqrt 2\big((k_1+k_1)^2 +k_1^2\big)\big)\Big|
$$
for some $\psi' \in \mathbb R$.
Proceeding as before, we obtain instead of (3.18)
the following bound on $|S|^2$
$$
\begin{aligned}
&\Big|\sum_{\ell_1, \ell_1', \ell_2, \ell_2'} \rho_1\Big(\frac{\ell_1}{L_1}\Big)\rho_2\Big(\frac {\ell_1'}{L_1'}\Big)
\rho_1\Big(\frac {\ell_2}{L_2}\Big)\rho_2 \Big(\frac {\ell_2'}{L_2'}\Big)
e(\ell.\psi') \, e\big(\sqrt 2\big((\ell_1+\ell_2)(\ell_1'+\ell_2')+\ell_1\ell_1'\big)\big)\Big|\\
&\leq \sum_{\ell_1, \ell_2} \rho_1 \Big(\frac {\ell_1}{L_1}\Big) \rho_1\Big(\frac {\ell_2}{L_2}\Big)
\Big|\sum_{\ell_1'} \rho_2 \big(\frac {\ell_1'}{L_1'}\Big)
\, e\big( \sqrt 2(2\ell_1+\ell_2)\ell_1'\big)\Big|. \Big|\sum_{\ell_2'} \rho_2 \Big(\frac {\ell_2'}{L_2'}\Big)
e\big(\sqrt 2 (\ell_1+\ell_2)\ell_2'\big)\Big|\\
&\leq \sum_{\substack{\ell_1\lesssim L_1\\ \ell_2\lesssim L_2}} \ \frac 1{L_1' \Vert\sqrt 2(2\ell_1+\ell_2)\Vert^2+\frac 1{L_1'}}\cdot
\frac 1{L_2' \Vert\sqrt 2(\ell_1+\ell_2)\Vert^2+\frac 1{L_2'}}.
\end{aligned}
$$
Assuming $L_1\geq L_2$, we obtain (performing first summation over $\ell_2$)
$$
\begin{aligned}
&\sum_{\substack{\ell_1\lesssim L_1\\ \ell_2 \lesssim L_2}} \ \frac 1{L_1' \Vert \sqrt 2 \ell_1 \Vert^2+\frac 1{L_1'}}\cdot
\frac 1{L_2' \Vert\sqrt 2(2\ell_1-\ell_2)\Vert ^2+\frac 1{L_2'}}\\
&< C(L_2+L_2') \sum_{\ell_1\lesssim L_1} \ \frac 1{L_1' \Vert\sqrt 2\ell_1\Vert^2+\frac 1{L_1'}}
<C(L_2+L_2')(L_1+L_1')< c M_1M_2
\end{aligned}
$$
proving (3.17).
\end{proof}

The next distributional considerations are very similar to those in \cite{B}.
Fix $\zeta \in 0B_3$.
We have by (3.5)
$$
e_{k_1, k_2} (\zeta)= \frac {\zeta_1^{k_1} \zeta_2^{k_2} \zeta_3^{N-k_1-k_2}} { (2k_1!k_2!(N-k_1-k_2)!)^{\frac 12}}
\sqrt{(N+2)!}.\eqno{(3.21)}
$$
Let us assume $N-k_1-k_2\asymp N$.
Otherwise, assuming say $k_2\asymp N$, we switch variables, writing $k_2=N-k_1-k_3$ and in this case
$$
e_{k_1, k_3}(\zeta) =\frac {\zeta_1^{\ell_1} \zeta_3^{k_3} \zeta_2^{N-k_1-k_3}}{(2k_1!k_3! (N-k_1-k_3)!)^{\frac 12}}
\sqrt{(N+2)!}.\eqno{(3.22)}
$$
Writing
$$
(3.21) = e\big(k_1\psi_1+k_2\psi_2+(N-k_1-k_3)\psi_3\big)
\ \frac {|\zeta_1|^{k_1}|\zeta_2|^{k_2}(1-\zeta _1^2 -\zeta^2_2)^{\frac {N-k_1-k_2} 2}}
{(2k_1! k_2! (N-k_1-k_2)!)^{\frac 12}} \ \sqrt{(N+2)!}
$$
for some $\psi_1, \psi_2, \psi_3$ (note that this first factor is harmless in view of the formulation of Lemma 4), we first need to
analyze the distribution of
$$
\frac {|\zeta_1|^{k_1}|\zeta_2|^{k_2} (1-\zeta_1^2 -\zeta_2^2)^{\frac {N-k_1-k_2}2}}
{(k_1!k_2!(N-k_1-k_2)!)^{\frac 12}} \ \sqrt{(N+2)!}.\eqno{(3.23)}
$$
Set $t_1=|\zeta_1|^2, t_2 =|\zeta_2|^2$.
By Stirling's formula
$$
\begin{aligned}
(3.23) &\sim \frac {t_1^{k_{1/2}} t_2^{k_{2/2}}(1-t_1-t_2)^{\frac {N-k_1-k_2}2} (N+2)^{\frac {N+2}2} (N+2)^{\frac 14}}
{k_1^{1/2} k_2^{1/2} (N-k_1-k_2)^{\frac {N-k_1-k_2}2} k_1^{\frac 14} k_2^{\frac 14} (N-k_1-k_2)^{\frac 14}}\\
&\sim\frac {N.N^{\frac 14}}{k_1^{1/4} k_2^{1/4} (N-k_1-k_2)^{\frac 14}} 
\Big(\frac {t_1}{k_1N^{-1}}\Big)^{k_{1/2}} \ \Big(\frac {t_2}{k_2N^{-1}}\Big)^{k_{2/2}}
\Big(\frac {1-t_1-t_2}{1-k_1 N^{-1}-k_2 N^{-1}}\Big)^{\frac {N-k_1-k_2}2}\end{aligned}
\eqno{(3.24)}
$$
and because of the normalization factor $\frac 1{\sqrt{N\frac {N+1}2}}$ in (3.12), (3.13), we may drop the $N$ factor in
(3.24).

Write for $u=t+\Delta u$, $\Delta u=o(t)$
$$
\Big(\frac tu\Big)^u = e^{-(t+\Delta u)(\frac {\Delta u}t -\frac 12 (\frac {\Delta u}t)^2+\cdots)}
= e^{-\Delta u -\frac {(\Delta u)^2}{2t}+\cdots}
$$
Hence (3.24) gives (after removal of the $N$-factor)
$$
\frac {N^{\frac 14}}
{k_1^{\frac 14} k_2^{\frac 14}(N-k_1-k_2)^\frac 14} 
\ e^{-\frac N2 [\frac {(k_1N^{-1} -t_1)^2}{t_1}+
\frac {(k_2N^{-1}-t_2)^2}{t_2} + \frac {((k_1+k_2)N^{-1}-t_1-t_2)^2}{1-t_1 -t_2}+\cdots]}\eqno{(3.25)}
$$
and the distribution in $(k_1, k_2)$-space localizes to
$$
\begin{cases}
&|k_1 -t_1 N|\lesssim\sqrt{t_1 N}\\
&|k_2 -t_2 N|\lesssim  \sqrt {t_2N}.
\end{cases}
\eqno{(3.26)}
$$
Thus, if we fix a center $\bar k = ([t_1 N], [t_2N])$ (3.26) corresponds to the tile
$$
Q_k =\{ k\in\Delta'\cup \Delta''; |k_1-\bar k_1|\lesssim \sqrt {\bar k_1}, |k_2-\bar k_2|\lesssim \sqrt {\bar k_2}\}.\eqno{(3.27)}
$$
Note that for (3.22), we obtain a tile with a different shape
$$
Q_{\bar k}' =\{ k\in\Delta'\cup\Delta''; |k_1-\bar k_1 |\lesssim \sqrt{\bar k_1}; |k_1+k_2-\bar k_1 -\bar k_2|\lesssim \sqrt {N-\bar k_1-
\bar k_2}\}\eqno{(3.28)}
$$
or with
$$
|k_1-\bar k_1|\lesssim \sqrt{\bar k_1} \text { replaced by $|k_2-\bar k_2|\lesssim \sqrt{\bar k_2}$}.\eqno{(3.29)}
$$

\centerline{
\begin{minipage}{1.5 in}
\centerline
{\hbox{\includegraphics[width = 3.5in]{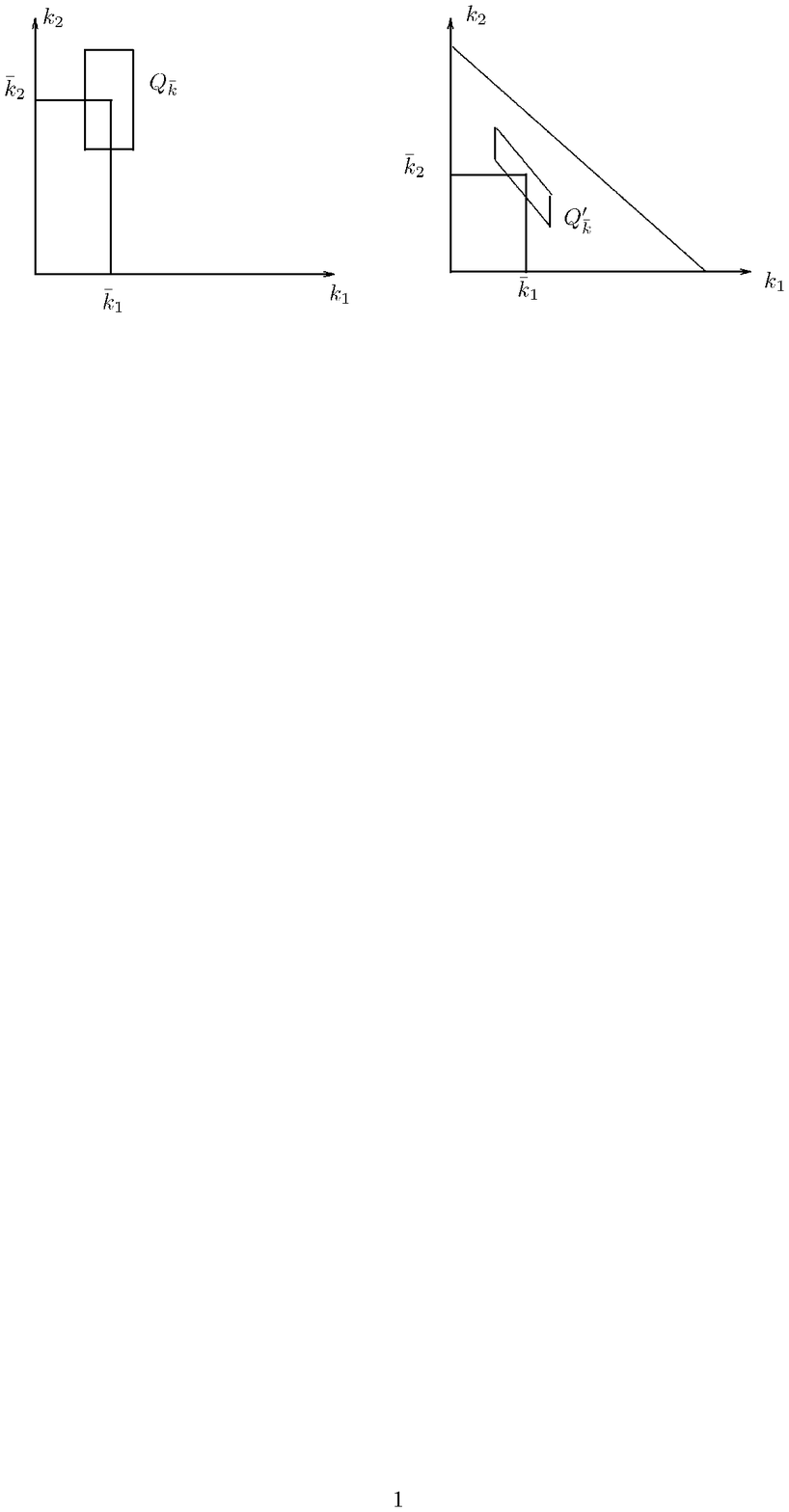}}}
\end{minipage}}

Going back to (3.12), (3.13), some attention is required due to the presence of the different factors
$$
e \Big(j_1 \frac {k_1}{N} +j_2\frac {k_2}N\Big) \ \text { and } \ e\Big(j_1 \frac {\tilde k_1}N + j_2\frac {\tilde k_2}N\Big)
$$
depending on whether $k\in \Delta'$ or $k\in\Delta''$.
Each argument is affine in $k$, but with a different expression.
Hence it is natural to apply Lemma 4 to the intersections
$$
Q_k \cap\Delta' \quad Q_k\cap \Delta''\eqno{(3.30)}
$$
$$
Q_k'\cap \Delta' \quad Q_k' \cap\Delta''.\eqno{(3.31)}
$$
The $Q_k$ appear when $N-k_1-k_2\asymp N$ and hence $Q_k\cap\Delta'$, $Q_k\cap\Delta''$ are still boxes.

If $Q_k'\cap \Delta'\not=\phi$, $Q_k' \cap\Delta''\not= \phi$ (assuming, as we may, that $N-k_1-k_2<\frac N{100}$),
clearly $k_1\approx \frac N2, k_2\approx \frac N2$ and $Q_k'$ has length $\sim\sqrt N$.
Thus we may then in either case (3.28), (3.29) take for $Q_{\bar k}'$ a box
$$
Q_{\bar k}' =\{ (k_1, k_2)\in\Delta'\cup\Delta''; |k_2-\bar k_2|\lesssim \sqrt N \ \text { and } \ |k_1+k_2-\bar k_1- \bar k_2|\lesssim
\sqrt{N-\bar k_1-\bar k_2}\}.
$$
The intersections $Q_{\bar k}'\cap \Delta', Q_{\bar k}'' \cap\Delta''$ have the same structure, i.e. $k_2$ and $k_1+k_2$ restricted to
suitable intervals.

Applying Lemma 4 to the summation for $k\in Q_{\bar k} \cap\Delta'$, $k\in Q_{\bar k}\cap \Delta''$ (after proper mollification as required
in Lemma 4) gives the bound
$$
C\frac {N^{\frac 14}}{\bar k_1^{\frac 14} \bar k_2^{\frac 14} (N- \bar k_1-\bar k_2)^{\frac 14}} (\bar k_1\bar k_2)^{\frac 14} <C
$$ 
and for $k\in Q_{\bar k}'\cap \Delta', k\in Q_{\bar k}' \cap\Delta''$ (assuming $\bar k_2\asymp N$)
$$
C\frac {N^{\frac 14}}{\bar k_1^{\frac 14} \bar k_2^{\frac 14} (N-\bar k_1 -\bar k_2')^{\frac 14}}
\big(\bar k_1(N-\bar k_1-\bar k_2)\big)^{\frac 14} < C.
$$
Similarly to \cite {B}, we perform a tiling of $\Delta' \cup\Delta''$ as dictated by (3.25) and exploit the exponentially decaying factors
to get a bounded collected contribution (the reader will easily check details).
This completes the proof of Proposition 3.

\end{document}